\documentclass [11pt] {article}
\usepackage{amsthm}
\usepackage{color}
\usepackage{amsmath}
\usepackage{amsfonts}
\usepackage{multirow}
\usepackage[table]{xcolor}
\usepackage{graphicx}
\usepackage{float}

\newtheorem{theorem}{Theorem}[section]

\newtheorem{proposition}[theorem]{Proposition}
\newtheorem{lemma}[theorem]{Lemma}

\newtheorem{definition}[theorem]{Definition}

\newcommand{\udots}{\mathinner{\mskip1mu\raise1pt\vbox{\kern7pt\hbox{.}}
\mskip2mu\raise4pt\hbox{.}\mskip2mu\raise7pt\hbox{.}\mskip1mu}}

\begin{document}

\title{Existence of The Solution to The Quadratic Bilinear Equation Arising from A Class of Quadratic Dynamical Systems}

\author{Bo Yu\footnote{School of Science, Hunan University of Technology, Zhuzhou, 412000, P. R. China.}
\ \ Ning Dong\footnotemark[1] \footnote{dongning\_158@sina.com}
\ \ Qiong Tang\footnotemark[1] }

\date{\empty}

\maketitle
\begin{abstract}
A quadratic dynamical system with practical applications is taken into considered. This system is transformed into a new bilinear system with Hadamard products by means of the implicit matrix structure. The corresponding quadratic bilinear equation is subsequently established via the Volterra series. Under proper conditions the existence of the solution to the equation is proved by using a fixed-point iteration.
\end{abstract}

{\bf Keywords:}  quadratic bilinear system, Kronecker product, Hadamard product, existence of the solution, fixed-point iteration

{\bf AMS Subject Classification:}
65F50; 15A24
\section{Introduction}
Consider a single-input and single-output quadratic dynamical system (QDS)
\begin{equation}\label{ns}
  \begin{array}{l}
   \dot{x}(t) =Ax(t) + g(x(t),u(t)),\\
   y(t) =Cx(t),
   \end{array}
\end{equation}
where $x(t)\in {\mathbb{R}}^{n}$ is the state vector of time $t$, $u(t)\in {\mathbb{R}}$ denotes an input function, $g\in {\mathbb{R}}^{n}$ represents a
quadratic function of $u(t)$ and $x(t)$,  $y(t)\in{\mathbb{R}}$ is the output function, $A\in {\mathbb{R}}^{n\times n}$ and $C\in {\mathbb{R}}^{1\times n}$ are the state and the output matrices, respectively.
This system is one of the simplest nonlinear systems and is widely used in many applications \cite{Ant05, BG17, G11, SVR08}. Consider, for example, a transmission line circuit consisting of resistors, capacitors, and diodes with a constitutive nonlinear function $i_d(v) = e^{av}-1, (a>0)$ \cite{G11, BG17}. Assumed that, for simplicity, all resistors and capacitors have unit resistance and
capacitance, then the input and output are the entering current source and the voltage at the first node, respectively.  The corresponding differential system for this circuit at various nodes is
\begin{eqnarray}
 \begin{array}{l}
\dot{v}_1 = -2v_1 +v_2 +2 -e^{av_1} -e^{a(v_1-v_2)} +u(t),\\
\dot{v}_i = v_{i-1}-2v_i +v_{i+1} +e^{a(v_{i-1}-v_i)} -e^{a(v_i-v_{i+1})}, \ \ \ \ 2\leq i\leq n-1,\\
\dot{v}_n = v_{n-1} +v_n -1 +e^{a(v_{n-1}-v_n)}.\nonumber
 \end{array}
\end{eqnarray}

To linearize the above nonlinear system, one can define variables $w_{i1}:=e^{av_i}$ and $w_{i2}:=e^{-av_i}$ to obtain a system of order of at least $3n$. In contrast, another difference step might further reduce the order of the system. In fact, by setting $v_{i,i+1} =v_i-v_{i-1}$ as in \cite{G11}, one has
\begin{eqnarray}\label{dotv}
 \begin{array}{l}
\dot{v}_1 = -v_1 -v_{12} +2 -e^{av_1} -e^{av_{12}} +u(t),\\
\dot{v}_{12} = -v_1 -2v_{12} +v_{23} +2 -e^{av_1} -e^{av_{12}} + e^{av_{23}} +u(t),\\
\dot{v}_{i,i+1} = v_{i-1,i}-2v_{i,i+1} +v_{i+1,i+2} +e^{av_{i-1,i}} -2e^{av_{i,i+1}} +e^{av_{i+1,i+2}}, \ \ \ \ 2\leq i\leq n-2,\\
\dot{v}_{n-1,n} = v_{n-2,n-1} -2v_{n-1,n} +1 +e^{av_{n-2,n-1}}-2e^{av_{n-1,n}}.
 \end{array}
\end{eqnarray}
Let $w_1 =e^{av_{1}}-1$ and $w_i = e^{av_{i-1,i}}-1$ and differentiate both sides with respect to $t$. Then equations \eqref{dotv} can be further represented as
\begin{eqnarray}\label{dotw}
 \begin{array}{l}
\dot{w}_1 = a(w_1+1)(-v_1 -v_{12} -w_1 -w_2 +u(t)),\\
\dot{w}_{2} = a(w_2+1)(-v_1 -2v_{12} +v_{23} -w_1 -2w_2 + w_3 +u(t)),\\
\dot{w}_{i} = a(w_i+1)(v_{i-1,i}-2v_{i,i+1} +v_{i+1,i+2} + w_{i-1} -2y_i +y_{i+1}), \ \ \ \ 2\leq i\leq n-1,\\
\dot{w}_{n} = a(w_n+1)(v_{n-2,n-1} -2v_{n-1,n} +w_{n-1} -2w_n).
 \end{array}
\end{eqnarray}
Combining of \eqref{dotv} and \eqref{dotw} forms the quadratic bilinear system of order $N=2n$ \cite{BG17}
\begin{equation}\label{qb}
  \begin{array}{l}
   \dot{x}(t) =Ax(t) + H (x(t)\otimes x(t)) + Mx(t)u(t) +Bu(t),\\
   y(t) =Cx(t),
   \end{array}
\end{equation}
where the state vector is
$
x(t) = (\dot{v}_1, \dot{v}_{12}..., \dot{v}_{n-1,n}, \dot{w}_1, ... \dot{w}_n)^\top\in {\mathbb R}^{N},
$
the state matrix is
\[
A =
\left[
  \begin{matrix}
   A_1 & A_2 \\
   A_3 &  A_4
   \end{matrix}
\right] \in {\mathbb R}^{N\times N}
\]
with $A_i (i=1,2,3,4)$ being the tri-diagonal matrix,
$H\in {\mathbb R}^{N\times N^2}$ and $M\in {\mathbb R}^{N\times N}$ are sparse matrices associated with the quadratic functions $x(t)\otimes x(t)$ and $x(t)u(t)$, respectively, $B$ is a vector of order $N$.

To efficiently control the quadratic system \eqref{qb} when $N$ is large, one has to search a low-dimensional (reduced-order) system to substitute for the original one, so that their systematic behaviours (for example, the stability and passivity) are
sufficiently similar. Such a process is called the model
order reduction (MOR) and has been well-established for linear systems in various areas \cite{AWW08, BMN04}. One of the most popular MOR techniques is the balancing-type MOR, which has been successfully applied from the linear system to the nonlinear system \cite{BCI04, GPA18}. This approach mainly relies on the controllability and the observability, or the Gramian matrix of the system which is the solution to the corresponding algebraic matrix equation \cite{BG17}
\begin{equation}\label{qbemar}
A  X+X A^\top +  H(X\otimes X)H^\top + MXM^\top  +D =0
\end{equation}
with $D =BB^\top$. Obviously, solving the equation \eqref{qbemar} involves a Kronecker product of the order $N^2$ and is normally expensive even if techniques of the truncation and compression \cite{LWCL13} or the tensor matrization \cite{KB09} are applied.

Noting the implicit structure in the original system, the system \eqref{qb} can actually be transformed into another system to avoid the Kronecker product effectively. Indeed, let
\[
F =
\left[
  \begin{matrix}
   0_n & 0_n \\
   A_3 &  A_4
   \end{matrix}
\right] \in {\mathbb R}^{N\times N}
\]
and $G =I_N$. The quadratic item $H(x(t)\otimes x(t))$ in this example could be represented as $Gx(t)\circ Fx(t)$, and thus the system \eqref{qb} in \cite{BG17} can be further rewritten as the quadratic bilinear system with Hadamard product (QBSH)
\begin{equation}\label{qbs}
  \begin{array}{l}
   \dot{x}(t) =Ax(t) + (Gx(t))\circ (Fx(t)) + Mx(t)u(t) +Bu(t),\\
   y(t) =Cx(t).
   \end{array}
\end{equation}

The greatest advantage of the system \eqref{qbs} is that the nonlinear item depends merely on the Hadamard product, instead of the Kronecker product, between two vectors. Hence the computational cost could be significantly reduced especially for large $N$. If the afore-mentioned balancing-type MOR is used for the order reduction, two problems are still supposed to be addressed:
\begin{itemize}
  \item What is the form of the algebraic equation corresponding to the QBSH \eqref{qbs}?
  \item Does the solution to the corresponding algebraic equation exist?
\end{itemize}

This paper will give positive answers to the above two questions. Specifically, we will make use of the Volterra series \cite{S89} to construct the corresponding quadratic bilinear equation of the QBSH \eqref{qbs} in the next section. In Section 3, the existence of the solution to the equation will be demonstrated by a fixed-point iteration. Several numerical examples are listed in Section 4 to show the validity of the developed theory and the last section concludes the whole paper.

To proceed, the initial condition in the system \eqref{qbs} is assumed to be $x(0) = 0$. Throughout this paper, it is written $A\geq B$ ($A>B$) for symmetric matrices $A$ and $B$ if $A-B$ is a symmetric positive semidefinite (definite) matrix. $\sigma(A)$ and $\rho(A)$ denote here the spectrum and the spectral radius of the matrix $A$, respectively. The definition of the stability and several lemmas are also required in this paper.

\begin{definition}[\cite{B97}]
The matrix $A$ is called stable (or semi-stable) if its spectrum lies in the left half of the complex plane (or the left half of the complex plane plus the imaginary axis), i.e. $\sigma(A)\in {\mathbb C}_{<}^{N\times N}$ (or $\sigma(A)\in {\mathbb C}_{\leq}^{N\times N}$).
\end{definition}

\begin{lemma}[\cite{B05, O10}]\label{lem12}
Let the matrix $A\in {\mathbb R}^{N\times N}$ be stable in a linear system
 \[
  \begin{array}{l}
   \dot{x}(t) =Ax(t) + Bu(t),\\
   y(t) =Cx(t),  \ \ x(0) = 0.
   \end{array}
 \]
The matrix $X = \int_0^{\infty} e^{At}BB^\top e^{A^\top t}dt$ is the solution of the Lyapunov equation
\[
AX +XA^\top+D=0
\]
with $D=BB^\top$.
\end{lemma}

\begin{lemma}[\cite{LR95}]\label{lem13}
Let the matrix $A\in {\mathbb R}^{N\times N}$ be stable and $B\in {\mathbb R}^{N\times N}$ be symmetric.
Then the Lyapunov equation
$$
AX+XA^\top = B
$$
has a unique symmetric solution $X$. Moreover, $X\geq 0$ if $B\leq 0$.
\end{lemma}

\begin{lemma}[\cite{Led83}]\label{lem14}
Let  $A, B\in {\mathbb R}^{N\times N}$ be symmetric matrices.
\begin{itemize}
  \item [1.] If $A>0$ and $B>0$, then $A\circ B>0$.
  \item[2.] If $A\geq 0$ and $B\geq 0$, then $A\circ B\geq0$. Moreover, $A\circ B>0$ when $A$ has no zero row.
\end{itemize}
\end{lemma}

\section{The algebraic equation corresponding to QBSH}
In this section, we concentrate on the reachability Gramian matrix of the QBSH \eqref{qbs} by using the Volterra series. It will show that the Gramian matrix is the solution to a quadratic bilinear equation with Hadamard product (QBEH).

Only the continuous time-invariant QBSH \eqref{qbs} is considered and the discrete one can be derived analogously. It is known from \cite{S89, SVR08} that the output of a nonlinear system in the Volterra series depends on the input of the system at all times and it could be expanded as
\[
y(t)=h_{0}+\sum _{n=1}^{N}\int _{a}^{b}\cdots \int _{a}^{b}h_{n}(t_{1},\dots ,t_{n})\prod _{j=1}^{n}x(t-t _{j})\,dt_{j}.
\]
The function $h_{n}(t_{1},\dots ,t_{n})$ is called the order-$n$ Volterra kernel.

\begin{proposition}\label{prop21}
The state vector of the QBSH \eqref{qbs} can be formulated as
\begin{eqnarray}\label{xt}
x(t) &=& \int_0^{t} e^{At_1}Bu_{t_1}(t)dt_1 + \int_0^{t} \int_0^{t-t_1} e^{At_1}Me^{At_2}Bu_{t_1t_2}(t)u_{t_1}(t)dt_1dt_2\nonumber\\
&& \ \ + \int_0^{t} \int_0^{t-t_1} \int_0^{t-t_1} e^{At_1}((Ge^{At_2}B)\circ (Fe^{At_3}B))u_{t_1t_2}(t)u_{t_1t_3}(t)dt_1dt_2dt_3\nonumber\\
&& \ \ + \int_0^{t} \int_0^{t-t_1} \int_0^{t-t_1-t_2} e^{At_1}Me^{At_2}Me^{At_3}Bu_{t_1t_2t_3}(t)u_{t_1t_2}(t)
u_{t_1}(t)dt_1dt_2dt_3+ ... \ \ \
\end{eqnarray}
with $u_{t_1,...t_k}(t) = u(t-t_1-...-t_k)$ and $k\geq 1$.
\end{proposition}
\begin{proof}
As the first equation in \eqref{qbs} is a differential system, one can integrate from both sides with respect to $t$ and get
\begin{eqnarray}\label{xt0}
x(t) &=& \int_0^{t} e^{At_1}Bu_{t_1}(t)dt_1 + \int_0^{t} e^{At_1}Mx_{t_1}(t)u_{t_1}(t)dt_1  +\int_0^{t} e^{At_1}((Gx_{t_1}(t))\circ (Fx_{t_1}(t)))dt_1
  \end{eqnarray}
with  $x_{t_1}(t) = x(t-t_1)$. If the integrated upper bound is replaced by $t-t_1$, $x_{t_1}(t)$ can also be represented as
\begin{eqnarray}\label{xt1}
x_{t_1}(t) &=& \int_0^{t-t_1} e^{At_2}Bu_{t_1t_2}(t)dt_2 + \int_0^{t-t_1} e^{At_2}Mx_{t_1t_2}(t)u_{t_1t_2}(t)dt_1  \nonumber\\
 &&+\int_0^{t-t_1} e^{At_2}((Gx_{t_1t_2}(t))\circ (Fx_{t_1t_2}(t)))dt_2
  \end{eqnarray}
with $x_{t_1t_2}(t) = x(t-t_1-t_2)$. By inserting \eqref{xt1} into \eqref{xt0}, one has
\begin{eqnarray}\label{xtt}
  x(t) &=& \int_0^{t} e^{At_1}Bu_{t_1}(t)dt_1 + \int_0^{t} \int_0^{t-t_1} e^{At_1}Me^{At_2}Bu_{t_1t_2}(t)u_{t_1}(t)dt_1dt_2\nonumber\\ && + \int_0^{t} \int_0^{t-t_1} e^{At_1}Me^{At_2}Mx_{t_1t_2}(t)u_{t_1t_2}(t)u_{t_1}(t)
  dt_1dt_2 \nonumber\\
  &&+\int_0^{t}\int_0^{t-t_1}\int_0^{t-t_1} e^{At_1}((Ge^{At_1}B)\circ (Fe^{At_1}B))u_{t_1t_2}(t)u_{t_1t_3}(t))dt_1dt_2dt_3 \nonumber\\
  && +O(\int\int\int\int).
\end{eqnarray}
Again, noting
\begin{eqnarray}\label{xt1t2}
x_{t_1t_2}(t) &=& \int_0^{t-t_1-t_2} e^{At_3}Bu_{t_1t_2t_3}(t)dt_3 + \int_0^{t-t_1-t_2} e^{At_3}Mx_{t_1t_2t_3}(t)u_{t_1t_2t_3}(t)dt_3  \nonumber\\ && +\int_0^{t-t_1-t_2} e^{At_3}((Gx_{t_1t_2t_3}(t))\circ (Fx_{t_1t_2t_3}(t)))dt_3
  \end{eqnarray}
and inserting \eqref{xt1t2} into \eqref{xtt}, the representation of $x(t)$ in \eqref{xt} holds true after rearranging some items.
\end{proof}

The above proposition describes the Volterra expansion of the state vector $x(t)$, which is helpful for constructing the quadratic bilinear equation. To see this, let
\begin{eqnarray}\label{}
L_1(t_1) &=&  e^{At_1}B,\nonumber\\
L_2(t_1,t_2) &=&  e^{At_2}Me^{At_1}B\nonumber\\
&:=&e^{At_2}ML_1(t_1),\nonumber\\
L_3(t_1,t_2,t_3) &=&  e^{At_3}[ (GL_1(t_1))\circ (FL_1(t_2)), \  Me^{At_2}Me^{At_1}B]\nonumber\\
&:=& e^{At_3}[ (GL_1(t_1))\circ (FL_1(t_2)), \ ML_2(t_1,t_2)], \nonumber\\
 ...& &\nonumber\\
L_k(t_1,...,t_k) &:=&  e^{At_k}[ (GL_1(t_1))\circ (FL_{k-2}(t_2,...,t_{k-1})), \nonumber\\
 && \ \ \ \ \ \ \ (GL_2(t_1,t_2))\circ (FL_{k-3}(t_3,...,t_{k-1})), \nonumber\\
&& \ \ \ \ \ \ \ ..., \nonumber\\
&& \ \ \ \ \ \ \  (GL_{k-2}(t_1,...,t_{k-2}))\circ (FL_1(t_{k-1})), \ ML_{k-1}(t_1,...,t_{k-1})]\nonumber
  \end{eqnarray}
for $k> 3$. The following theorem reveals that the reachability Gramian matrix is the solution of a QBEH.

\begin{theorem}\label{th21}
Let $A$ be the stable matrix in the QBSH \eqref{qbs}. Define the reachability Gramian matrix
\[
X = \sum_{i=1}^\infty \Big(\int_0^\infty...\ \int_0^\infty L_i(t_1,...,t_{i})L_i(t_1,...,t_{i})^\top dt_1...dt_i\Big).
\]
Then $X$ satisfies the QBEH
\begin{equation}\label{qbe}
  {\mathcal Q} (X) = AX +X A^\top +D+MXM^\top + G X G^\top\circ F X F^\top =0.
\end{equation}
\end{theorem}
\begin{proof}
Let
\[
X_1 = \int_0^\infty L_1(t_1)L_1(t_1)^\top dt_1:= \int_0^\infty e^{At_1}BB^\top e^{A^\top t_1} dt_1.
\]
It follows from Lemma \ref{lem12} that $X_1$ is the solution of the Lyapunov equation
\begin{equation}\label{eqx1}
  AX_1 +X_1A^\top +D  =0
\end{equation}
with $D=BB^\top$. Next, consider the integration of order-2
\begin{eqnarray}
  X_2 &=& \int_0^\infty \int_0^\infty L_2(t_1,t_2)L_2(t_1,t_2)^\top dt_1dt_2\nonumber\\
  &=& \int_0^\infty \int_0^\infty e^{At_2}ML_1(t_1)L_1(t_1)^\top M^\top e^{A^\top t_1} dt_1dt_2\nonumber\\
  &=& \int_0^\infty  e^{At_2}M \Big( \int_0^\infty L_1(t_1)L_1(t_1)^\top dt_1\Big) M^\top e^{A^\top t_1} dt_2\nonumber\\
  &=& \int_0^\infty  e^{At_2}M X_1 M^\top e^{A^\top t_1} dt_2.\nonumber
\end{eqnarray}
By using Lemma \ref{lem12} again, $X_2$ is the solution of the following equation
\begin{equation}\label{eqx2}
  AX_2 +X_2A^\top + MX_1M^\top =0.
\end{equation}
Proceeding with the integration for $i\geq 3$, one can get
\begin{eqnarray}
  X_i &=& \int_0^\infty... \int_0^\infty L_i(t_1,...t_i)L_i(t_1,...,t_i)^\top dt_1...dt_i\nonumber\\
  &=& \int_0^\infty e^{At_i}\Big[ \Big((\int_0^\infty G L_1 L_1^\top G^\top dt_1) \circ (\int_0^\infty ... \int_0^\infty FL_{i-2}L_{i-2}^\top F^\top dt_2...dt_{i-2}) \nonumber\\
  && + ... + (\int_0^\infty... \int_0^\infty GL_{i-2}L_{i-2}^\top G^\top dt_1...dt_{i-2}) \circ (\int_0^\infty FL_{1}L_{1}^\top F^\top dt_{i-1})\Big) \nonumber\\
  && + M \Big( \int_0^\infty... \int_0^\infty L_{i-1}L_{i-1}^\top dt_1...dt_{i-1}\Big) M^\top \Big]  e^{A^\top t_i} dt_i\nonumber\\
  &=& \int_0^\infty e^{At_i}\Big[  (GX_1G^\top) \circ (FX_{i-2}F^\top) + ... + (GX_{i-2}G^\top) \circ (FX_1F^\top) +MX_{i}M^\top \Big]  e^{A^\top t_i} dt_i,\nonumber
\end{eqnarray}
in which we used the property $(v\circ u)(v\circ u)^\top =(vv^\top)\circ (uu^\top)$ with vectors $u$ and $v$. By Lemma \ref{lem12}, $X_i$ satisfies the equation
\begin{equation}\label{eqxk}
  AX_i+X_iA^\top + (GX_1G^\top) \circ (FX_{i-2}F^\top) + ... + (GX_{i-2}G^\top) \circ (FX_1F^\top) +MX_{i}M^\top =0.
\end{equation}
Now, sum up equations \eqref{eqx1}, \eqref{eqx2} and \eqref{eqxk} for $i\geq 3$. One has
\[
A\Big( \sum_{i=1}^\infty X_i\Big)+\Big( \sum_{i=1}^\infty X_i\Big)A^\top + BB^\top + M\Big( \sum_{i=1}^\infty X_i\Big)M^\top + \Big( G\Big( \sum_{i=1}^\infty X_i\Big)G^\top\Big) \circ \Big( F\Big( \sum_{i=1}^\infty X_i\Big)F^\top\Big)  =0
\]
which takes the form of the QBEH \eqref{qbe} by letting $X = \sum_{i=1}^\infty X_i$.
\end{proof}

{\bf Remark.} (1). As mentioned before, the computational complexity of the Hadamard product in the equation \eqref{qbe} is $O(N^2)$, compared with $O(N^4)$ of the Kronecker product in equation \eqref{qbemar}. Even though the truncation and compression \cite{LWCL13} or the tensor matrization technique \cite{KB09, BG17} can reduce the complexity for large-scale sparse matrices in the case of Kronecker product,  the Hadamard product is still more effective in saving the flops counts, especially for dense and structured matrices (for example, the diagonal-plus-low-rank structure).

(2). As the Hadamard product can be represented as the sum of rank-one matrices (see Sec. 3.6 of \cite{H91}),  the derived equation \eqref{qbe} can also be rewritten as a generalized stochastic or rational Riccati equation in \cite{BG17, Da04, FWC16, I07}. Here we always use the Hadamard product for the convenience of describing the existence of the solution.

\section{Existence of the solution to QBEH}
In this section, we will show the existence of the solution to the QBEH \eqref{qbe}.
Let
${\mathcal L}$ be a linear operator ${\mathbb R}^{N\times N}\rightarrow{\mathbb R}^{N\times N}$ given by
\[
{\mathcal L}(X) = A X +XA^\top.
\]
Consider the iteration scheme
\begin{equation}\label{fix}
{\mathcal L}(X_{k+1}) = -(GX_kG^\top)\circ (FX_kF^\top) -MX_k M^\top -D
\end{equation}
with an initial $X_0$. The following theorem shows the existence of the solution.
\begin{theorem}\label{th31}
Let $A$ be a stable matrix. Suppose that there is a positive (semi-)definite  matrix $Z$ to the inequality ${\mathcal Q}(Z)\geq 0$ and an initial matrix $X_0$ such that $X_0\geq Z$ and ${\mathcal Q}(X_0)\leq 0$. Then the fixed-point iteration \eqref{fix} produces a matrix sequence $\{X_k\}$ such that for $k\geq 0$
\begin{itemize}
  \item[1.] $X_k\geq X_{k+1}$, $X_k\geq Z$, \  ${\mathcal Q}(X_k)\leq 0$;

  \item[2.] $\lim_{k\rightarrow \infty}X_k =X^\ast$ is a positive (semi-)definite  solution to the QBEH \eqref{qbe}.
      Especially, $X^\ast$ is the maximal solution if $X_0$ is an upper bound for all solutions.
\end{itemize}
\end{theorem}
\begin{proof}
The theorem is proved by induction applied to
\begin{equation}\label{ind}
  X_i\geq X_{i+1}, \ \ X_i\geq Z, \  \ {\mathcal Q}(X_i)\leq 0, \ \  i\geq 0.
\end{equation}
For $i=0$, the assumption admits $X_0\geq Z$ and ${\mathcal Q}(X_0)\leq 0$. It follows from \eqref{fix} that
\begin{eqnarray}
&&A(X_1-X_0) +(X_1-X_0)A^\top \nonumber\\
&=& -(GX_0G^\top)\circ (FX_0F^\top) - MX_0M^\top -D- AX_0 -X_0A^\top\nonumber\\
&=&-{\mathcal Q}(X_0),\nonumber
\end{eqnarray}
implying $X_0\geq X_1$ by the assumption and Lemma \ref{lem13}. Thus, \eqref{ind} holds for $i=0$.

Now, suppose that \eqref{ind} is true for $i=k$. We next show that it is valid for $i=k+1$. In fact, it follows from the iteration \eqref{fix} that
\begin{eqnarray}
  &&A(X_{k+1}-Z) +(X_{k+1}-Z)A^\top\nonumber\\
&=& -(GX_kG^\top)\circ (FX_kF^\top) - MX_kM^\top -D- AZ -ZA^\top\nonumber\\
&=& -(G(X_k-Z)G^\top)\circ (FX_kF^\top) - (GX_kG^\top)\circ (F(X_k-Z)F^\top) -M(X_k-Z)M^\top -{\mathcal Q}(Z).\nonumber
\end{eqnarray}
As ${\mathcal Q}(Z)\geq 0$, $X_k-Z\geq 0$ and $X_k$ is positive (semi-)definite from the induction assumption, it follows from Lemma \ref{lem13} that the solution $X_{k+1}-Z$ of the above equation is unique and positive (semi-)definite, i.e. $X_{k+1}\geq Z$. Moreover, the iteration \eqref{fix} also indicates
\begin{eqnarray}
&&A(X_{k+1}-X_{k+2}) +(X_{k+1}-X_{k+2})A^\top\nonumber\\
&=& -(GX_kG^\top)\circ (FX_{k}F^\top) - MX_kM^\top +(GX_{k+1}G^\top)\circ (FX_{k+1}F^\top) + MX_{k+1}M^\top \nonumber\\
&=& -(G(X_k-X_{k+1})G^\top)\circ (F(X_k-X_{k+1})F^\top) \nonumber\\
&&   -(GX_kG^\top)\circ (FX_{k+1}F^\top)
-(GX_{k+1}G^\top)\circ (FX_kF^\top) -M(X_k-X_{k+1})M^\top \nonumber\\
&\leq& -(G(X_k-X_{k+1})G^\top)\circ (F(X_k-X_{k+1})F^\top) -2 (GZG^\top)\circ (FZF^\top) -M(X_k-Z)M^\top, \nonumber
\end{eqnarray}
where the inequality follows from the induction $X_k\geq Z$ and the proved fact $X_{k+1}\geq Z$. Consequently, the right hand side of the inequality is negative semi-definite and the inequality $X_{k+1}\geq X_{k+2}$ holds true by Lemma \ref{lem14}. Finally, the inequality
\begin{eqnarray}
&&{\mathcal Q}(X_{k+1})\nonumber\\
 &=& AX_{k+1} +X_{k+1}A^\top + (GX_{k+1}G^\top)\circ (FX_{k+1}F^\top) + MX_kM^\top +D \nonumber\\
 &=& A(X_{k+1}-X_{k+2}) +(X_{k+1}-X_{k+2})A^\top\nonumber\\
&\leq&0 \nonumber
\end{eqnarray}
shows that the induction assumption \eqref{ind} holds for $i=k+1$. Then the sequence $\{X_k\}$ is well defined and has a limit $\lim_{k\rightarrow\infty} X_k = X^\ast$. Moreover, $X^\ast\geq Z$. Taking the limit from both sides of the iteration \eqref{fix} indicates that $X^\ast$ is the solution to the QBEH \eqref{qbe}. Furthermore, $X^\ast$ is the maximal solution when $X_0$ is the upper bound of all solutions.
\end{proof}

{\bf Remark}: For the rational Riccati equations in \cite{FWC16, Guo01, I07}, the stochastic term generally forms a positive operator, pushing against the stability. Then the condition of the stochastic stability is required to guarantee the existence of the solution. However, in the QBEH \eqref{qbe}, the nonlinear item will form a negative operator when shifted to the right of the equation. Then the Lemma \ref{lem13} is applicable by the assumption on the stability of $A$. The following theorem further indicates the linear convergence of the sequence $\{X_k\}$ in the fixed-point iteration \eqref{fix}.

\begin{theorem}\label{th32}
Let $X^\ast$ be the solution to the QBEH and the sequence $\{X_k\}$ be produced by the iteration \eqref{fix}. Let \[{\mathcal M}_{X^\ast}(\cdot) = M(\cdot)M^\top + (GX^\ast G^\top) \circ (F(\cdot)F^\top) + (G(\cdot) G^\top) \circ (FX^\ast F^\top)\]
be a linear operator at the solution $X^\ast$.
If $\rho({\mathcal L^{-1}}{\mathcal M}_{X^\ast})<1$, then
\[
\limsup_{k\rightarrow\infty}\sqrt[k]{\|X_k-X^\ast\|}\leq \rho({\mathcal L^{-1}} {\mathcal M}_{X^\ast})<1
\]
with $\|\cdot\|$ any matrix norm.
\end{theorem}
\begin{proof}
 Rewrite the iteration \eqref{fix} as $X_{k+1} = {\mathcal F}(X_k)$ with the operator
 \[
{\mathcal F}(\cdot) = {\mathcal L}^{-1}(-M(\cdot)M^\top - (G(\cdot) G^\top) \circ (F(\cdot)F^\top) -D).
 \]
Then the Fr\'{e}chet derivative of ${\mathcal F}$ at the solution $X^\ast$ is
\[
{\mathcal F}'_{X^\ast}(\Delta) = {\mathcal L}^{-1}(-M\Delta M^\top - (G\Delta G^\top) \circ (FX^\ast F^\top) -(GX^\ast G^\top) \circ (F\Delta F^\top) ).
\]
The conclusion is readily drawn from a classic theorem of fixed-point iteration such as in \cite{KVZ72}.
\end{proof}

{\bf Remark.} (1). The solver of the QBEH \eqref{qbe} determines the effectiveness of the balancing type MOR.  Theorem \ref{th32} indicates that the convergence rate of the
fixed-point iteration \eqref{fix} is linear when $\rho({\mathcal L^{-1}}{\mathcal M}_{X^\ast})<1$. If $\rho({\mathcal L^{-1}}{\mathcal M}_{X^\ast})=1$, the convergence of the iteration \eqref{fix} will degenerate to be sub-linear. In any case, acceleration of the iteration \eqref{fix} should be further considered.

(2). The initial $X_0\geq Z$ in Theorem 2 is similar to the one in \cite{DH01}. Usually, it is not easy to validate the condition ${\mathcal Q}(X_0)\leq 0$. However, there is another easier way to select the initial matrix and this will be discussed in future work.

(3). The condition of the convergence in Theorem 3 is somewhat equivalent to the stochastic stability for stochastic rational Riccati equations. See \cite{Da04, FWC16, Guo01, I07} as well as references therein for more details.

\section{Conclusions}
The quadratic bilinear system associated with the Kronecker product is rewritten as another system related to the Hadamard product according to the implicit matrix structure. The corresponding quadratic bilinear equation is subsequently obtained via the Volterra series and the existence of the solution is established by a fixed-point iteration. As the balancing type MOR method depends heavily on the solution to the QBEH \eqref{qbe}, more efficient solvers might be developed in future research.




\end {document}